\documentclass[11pt]{amsart}

\usepackage{amsfonts, amsthm, amssymb, amsmath, stmaryrd, etoolbox}
\usepackage{comment}
\usepackage{mathtools}
\usepackage{graphicx,caption,subcaption}

\usepackage[inline]{enumitem}
\setlist{itemsep=0em, topsep=0em, parsep=0em}
\setlist[enumerate]{label=(\alph*)}
\usepackage{tikz}
\usetikzlibrary{matrix,arrows,external}
\definecolor{mycolor}{rgb}{0,0,0.7}
\usepackage{hyperref}
\hypersetup{colorlinks,linkcolor={mycolor},citecolor={mycolor},urlcolor={mycolor}}

\renewcommand{\epsilon}{\varepsilon}

\newcommand{\op}[1]{\operatorname{#1}}
\newcommand{\cat}[1]{\mathbf{#1}}
\renewcommand{\t}[1]{\textup{#1}}

\newcommand{\from}{\colon}

\renewcommand{\span}{\xrightarrow{\mathit{sp}}}
\newcommand{\cospan}{\xrightarrow{\mathit{csp}}}

\newcommand{\csC}{\widetilde{\mathbf{C}}}

\newcommand{\diagram}[1]{\raisebox{-0.5\height}{\includegraphics{#1}}}

\DeclareMathOperator{\id}{id}
\DeclareMathOperator{\ob}{Ob}

\DeclareMathOperator{\Sub}{Sub}

\newtheorem{thm}{Theorem}[section]
\newtheorem{lem}[thm]{Lemma}

\theoremstyle{remark}

\theoremstyle{definition}
\newtheorem{ex}[thm]{Example} 
\newtheorem{defn}[thm]{Definition}

\begin{document}

%\tableofcontents

\begin{abstract}
	We explore the notion of a span of cospans and define, for them, horizonal and vertical composition.  These compositions satisfy the interchange law if working in a topos $\cat{C}$ and if the span legs are monic. A bicategory is then constructed from $\cat{C}$-objects, $\cat{C}$-cospans, and doubly monic spans of $\cat{C}$-cospans. The primary motivation for this construction is an application to graph rewriting.
\end{abstract}

\title{Spans of cospans}
\author{Daniel Cicala}
\maketitle

%%%%%%%%%%%%%%%%%%%%%%%%%%%%%%%%%%%%%%%%%%%%%%%%%%%%%%%%%%%%%%%%%%%%%%
%
\section{Introduction} %                                INTRODUCTION
%
%%%%%%%%%%%%%%%%%%%%%%%%%%%%%%%%%%%%%%%%%%%%%%%%%%%%%%%%%%%%%%%%%%%%%%%%

There is currently interest in studying complex networks through the simpler networks of which they are comprised. This point of view is known as \textit{compositionality}. Various flavors of graphs (directed, weighted, colored, etc.) play an important role in this program because they are particularly well suited to model networks \cite{Baez_CompFrameMarkovProcess,Baez_CompFrameLinearNetworks,RoseSabadinWalters_SepAlgNCospansGraphs,RoseSabadinWalters_CalcColimsComp}. By adding a bit of structure to graphs, we can decide how to glue graphs together to make larger graphs.  

The structure that we want to consider is given by choosing two subsets of nodes, named \textit{inputs} and \textit{outputs}. When the inputs of one graph equal the outputs of another, we can glue the graphs together. One method for adding this structure is to use a cospan of graphs $I \to G \leftarrow O$ where $I$ and $O$ are discrete. Then gluing graphs together becomes a matter of composing cospans, which B\'{e}nabou \cite{Benabou_Bicats} described using pushouts in the context of a bicategory whose morphisms are cospans and $2$-morphisms are maps of cospans. Rebro \cite{Rebro_Span2} extended this idea to give a bicategory whose $2$-morphisms are cospans of cospans.  See the figure below to see what this means.    

\begin{figure}
	\centering	
	\begin{minipage}[b]{0.3\textwidth}
	\[
		\diagram{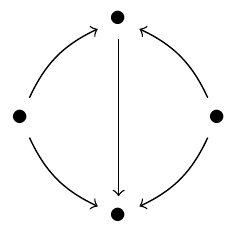}
	\]
	\subcaption{Map of cospans}
	\label{fig.MapOfCospans}
	\end{minipage}
	\begin{minipage}[b]{0.3\textwidth}
	\[
		\diagram{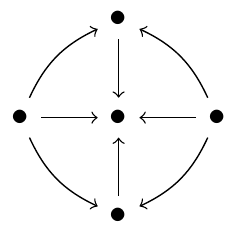}
	\]
	\subcaption{Cospan of cospans}
	\label{fig.CospanOfCospans}
	\end{minipage}
	\begin{minipage}[b]{0.3\textwidth}
	\[
		\diagram{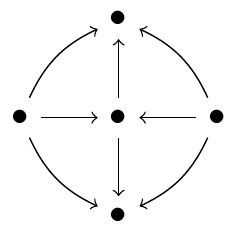}
	\]
	\subcaption{Span of cospans}
	\label{fig.SpanOfCospans}
\end{minipage}
\caption{Various morphisms of cospans}
\end{figure}

The goal of this paper is to explore another tool to study networks: spans of cospans (see Figure \ref{fig.SpanOfCospans}). Grandis and Par\'{e} studied these in the context of intercategories and showed that lax interchange held.  We will build a bicategory from certain spans of cospans.  Both B\'{e}nabou and Rebro only required sufficient colimits to compose maps of cospans and cospans of cospans, respectively. For us, we require sufficient limits and colimits while also ensuring they play well together. For this reason, we start with a topos $\cat{C}$, then let our $0$-cells be the objects of $\cat{C}$, $1$-cells the cospans in $\cat{C}$, and $2$-cells isomorphism classes of spans (with monic legs) of cospans. Our main theorem is that this construction, named $\cat{MonSp(Csp(C))}$, really gives a bicategory. 

The specific motivation for this particular construction is to create a framework to house the rewriting of graphs (with inputs and outputs).  Applying our construction to the topos $\cat{Graph}$, we consider $\cat{Rewrite}$, the full sub-bicategory of $\cat{MonSp(Csp(Graph))}$ whose objects are the discrete graphs. The $2$-cells of $\cat{Rewrite}$ represent all possible ways to rewrite one graph into another that respect the inputs and outputs. In this paper, graph rewrites are performed using the double pushout method. This is explained in Section \ref{sec.Rewriting}, though \cite{Ehrig_GraphGramAlgAp} and \cite{LackSoboc_AdhesiveCategories} contain more detailed accounts. Here is an example of what a $2$-cells in $\cat{Rewrite}$ will look like:
\[
	\label{NoRef_IntroRewrite2Cell}
	\diagram{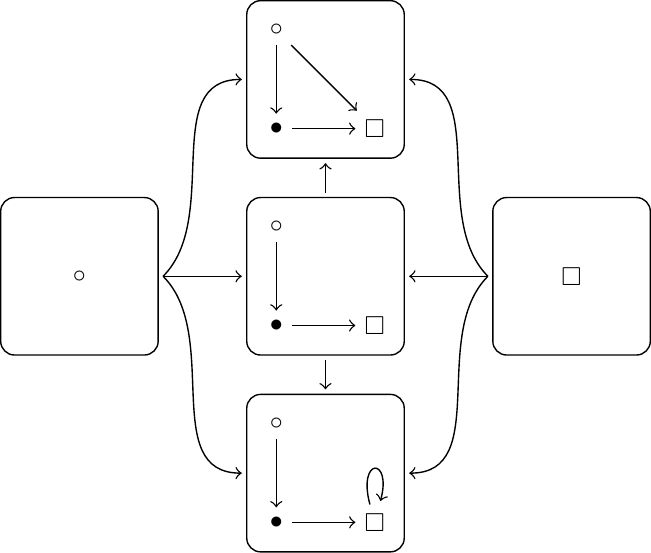}
\]
Inputs are depicted with `$\circ$', outputs with `$\square$', and nodes depicted with `$\bullet$' are neither. The picture above illustrates the rewriting of the top graph to the bottom graph by deleting an edge and adding a loop. 

The structure of this paper is as follows. We begin Section \ref{sec.SpansOfCospans} by defining spans of cospans and two ways to compose them. After declaring that we will be working in a topos and that the legs of the spans of cospans will be monic, we show that the compositions satisfy an interchange law.  This section ends with a construction of a bicategory whose $2$-cells are spans of cospans. Strictly speaking, there is no need for the full strength of a topos, so in Section \ref{sec.Disc on Assump}, we discuss how we might weaken our assumptions so that interchange still holds.  Then in Section \ref{sec.Rewriting} we give a brief introduction to graph rewriting and present our motivating example: $\cat{Rewrite}$, a bicategory that contains all possible ways to rewrite graphs.

The author would like to thank John Baez for many helpful discussions as well as Blake Pollard and Jason Erbele for contributing the Boolean algebra counterexample in Section \ref{sec.Disc on Assump}.
%
%
%
%
%
%
%
%
%
%
%
%
%%%%%%%%%%%%%%%%%%%%%%%%%%%%%%%%%%%%%%%%%%%%%%%%%%%%%%%%%%%%%%%%%%%%%%%%%%%%
%
\section{Spans of cospans} %                           SPAN OF COSPANS
\label{sec.SpansOfCospans}
%
%%%%%%%%%%%%%%%%%%%%%%%%%%%%%%%%%%%%%%%%%%%%%%%%%%%%%%%%%%%%%%%%%%%%%%%%%%%

We begin by introducing spans of cospans.  Given that our motivation is to have these be $2$-cells in some bicategory, we also show how to compose them.  In fact, there are two ways to do so and these will correspond to what will eventually be horizontal and vertical composition. Of course, we would like for these compositions to play nicely together and so we finish this section by showing that an interchange law holds, under certain assumptions. 

Before we begin, a few remarks on convention are in order. First, throughout this paper, tailed arrows ``$\rightarrowtail$" refer to monics, and two headed arrows ``$\twoheadrightarrow$" to quotient maps. Also, hooked arrows ``$\hookrightarrow$'' are canonical inclusions, which will be labeled $\iota_x$ when its codomain is $x$. To declutter the diagrams, only arrows that will eventually be referenced are given names. Spans $A \leftarrow B \to C$ are denoted by $B \from A \span C$ and cospans $X \to Y \leftarrow Z$ by $Y \from X \cospan Z$. This notation is vague, but context should dispense with any confusion. Note that, when first defining spans of cospans below, the object names might seem to be oddly chosen. The intention is to develop a consistency that will carry into the proof of the interchange law, at which point, the naming will seem more methodical.
%
%
%
%
%
%%%%%%%%%%%%%%%%%%%%%%%%%%%%%%%%%%%%%%%%%%%%%%%%%%%%%%%%%%%%%%%%%%%%%%
%
\subsection{Spans of cospans and their compositions} %      COMPOSITIONS
%
%%%%%%%%%%%%%%%%%%%%%%%%%%%%%%%%%%%%%%%%%%%%%%%%%%%%%%%%%%%%%%%%%%%%%

Suppose that we are working in a category $\cat{C}$. Given a parallel pair of cospans $L$, $S \from X \cospan Y$, a \emph{span of cospans} is a span $S' \from L \span S$ such that
\[
	\label{NoRef_SpanOfCospanDef}
	\diagram{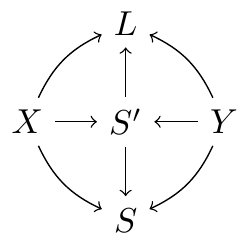}
\]
commutes. Given spans of cospans $S'_1,S'_2 \from L \span S$, then a morphism of spans of cospans is a $\cat{C}$-morphism $S'_1 \to S'_2$ such that the diagram
\begin{equation} 
	\label{diag.2-cell iso}
	\diagram{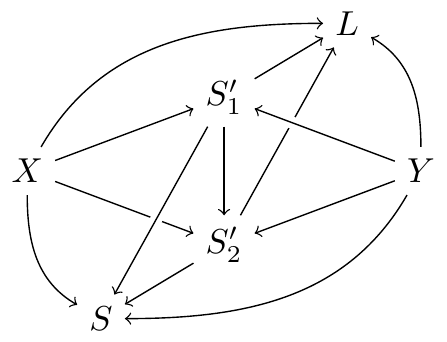}
\end{equation}
commutes. This is an isomorphism of spans of  cospans exactly when the  $\cat{C}$-morphism is.

There are two different ways to turn compatible pairs of spans of cospans into a single span of cospans.  Actually, we are foreshadowing that spans of cospans will soon be $2$-cells in a bicategory. Instead of beating around the bush, we immediately name these assignments for what they are: vertical and horizontal composition. As we will see, $\cat{C}$ should have enough limits and colimits for the compositions to be defined. Also, for the present moment, compositions will only be defined up to isomorphism. We will also hold off on looking for the typical properties composition should satisfy until we introduce the bicategorical structure. 

Take a pair of spans of cospans $S' \from L \span S$ and $S'' \from S \span L'$. Define \textit{vertical composition} $\circ_v$ by 
\begin{equation}
\label{eq.VertComp}
	S'' \circ_v S' \coloneqq S' \times_S S'' \from L \span L'.
\end{equation}
Diagrammatically, this is
\[
	\label{NoRef_VertCompDef}
	\diagram{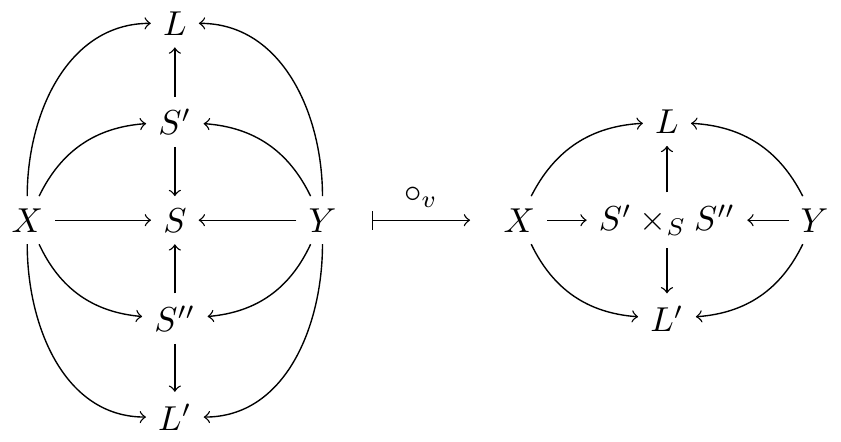}
\]
Now, let $L,S \from X \cospan Y$ and $R,T \from Y \cospan Z$ be cospans and let $S' \from L \span S$ and $T' \from R \span T$ be spans of cospans.  Define the assignment \textit{horizontal composition} $\circ_h$ by 
\begin{equation}
\label{eq.HorComp}
	T' \circ_h S' \coloneqq S' +_Y T'' \from L +_Y R \span S +_Y T,
\end{equation} 
which corresponds to 
\[
	\label{NoRef_HorCompDef}
	\diagram{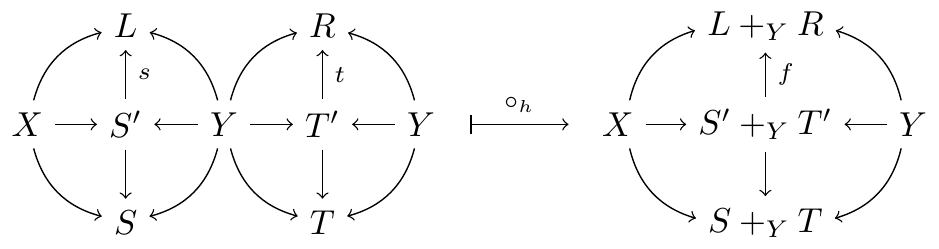}
\]
At this point, it is natural to ask whether the interchange law holds between vertical and horizontal composition. It does, but not without some further assumptions. 
%
%
%
%
%
%%%%%%%%%%%%%%%%%%%%%%%%%%%%%%%%%%%%%%%%%%%%%%%%%%%%%%%%%%%%%%%%%%%%%
%
\subsection{The interchange law}  %                     INTERCHANGE LAW
%
%%%%%%%%%%%%%%%%%%%%%%%%%%%%%%%%%%%%%%%%%%%%%%%%%%%%%%%%%%%%%%%%%%%%%

Let $\cat{C}$ be a topos with chosen pushouts and let both legs of each span of cospans be monic. To see examples of where the interchange law fails without these assumptions, see Section \ref{sec.Disc on Assump}.

The first thing we want to do is to show that the vertical and horizontal compositions are well-defined, up to isomorphism. To this end, we give a lemma that will be put to work several times during the course of this section.
%
%
%
%
%
%
%-----------LEMMA--------------

\begin{lem} 
	\label{lem.helpful little lemma}
	Given a diagram
	\begin{equation} 
		\label{diag.helpful little lemma}
		\diagram{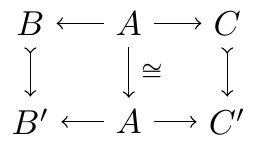}
	\end{equation}
	we get a pushout
	\begin{equation}
		\label{diag.helpful pushout}
		\diagram{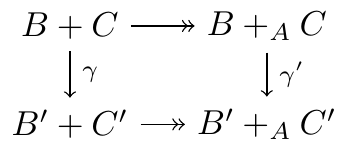}
	\end{equation}
	such that the canonical arrows $\gamma$ 
	and $\gamma'$ are monic.
\end{lem}
\begin{proof}
	Using its universal property, we see that 
	$\gamma$ factors through $B'+C$ as seen in 
	diagram
	\[
		\label{NoRef_HelpfulLemmaProof1}
		\diagram{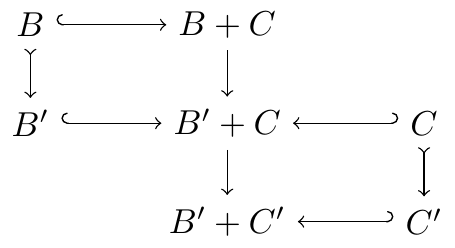}
	\]
	It is straightforward to check that the squares are both pushouts. By Lemma \ref{lem.adhesive properties}, we get that $\gamma$ must be monic and also that the monotonicity of $\gamma'$ will follow once \eqref{diag.helpful pushout} is shown to be a pushout.
	
	One can check that the right hand square commutes by using the universal property of $B+C$. To see that this square is a pushout, set up a cocone $D$
	\begin{equation} 
		\label{diag.helpful lemmma cocone}
		\diagram{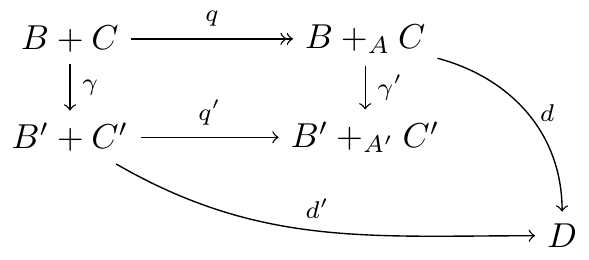}
	\end{equation}
	Then $d'\iota_{B'}$, $d'\iota_{C'}$, and $D$ form a cocone under the span $B' \leftarrow A' \to C'$ on the bottom face of diagram \eqref{diag.helpful little lemma}. This induces the canonical map $\gamma'' \from B'+_AC' \to D$.  It follows that $d'\iota_{B'}=\gamma'' q' \iota_{B'}$ and $d'\iota_{C'}=\gamma'' q' \iota_{C'}$. Therefore, $d'=\gamma'' q'$ by the universal property of coproducts.
	
	Furthermore, $dq\iota_B$, $dq\iota_C$, and $D$ form a cocone under the span $B \leftarrow A \to C$ on the top face of diagram \eqref{diag.helpful little lemma}. Then $dq\iota_B = d'\gamma\iota_B = \gamma'' q'\gamma\iota_B = \gamma'' \gamma' q \iota_B$ and $dq\iota_C = d'\gamma\iota_C = \gamma'' q'\gamma\iota_C = \gamma'' \gamma' q \iota_C$ meaning that both $d$ and $\gamma''\gamma'$ satisfy the canonical map $B+_AC \to D$.  Hence $d=\gamma''\gamma'$. 
	
	The universality of $\gamma''$ with respect to diagram \eqref{diag.helpful lemmma cocone} follows from the universality of $\gamma''$ with respect to $B'+_AC'$.
\end{proof}
%
%
%
%
%
%
%-------LEMMA: COMPOSITION RESP. MONICS------

\begin{lem}
	Vertical and horizontal composition of 
	spans of cospans respects monics.
\end{lem}
\begin{proof}
	The result for vertical composition 
	follows from the fact that pullbacks 
	respect monics. The result for horizontal 
	composition follows from applying Lemma 
	\ref{lem.helpful little lemma} to the 
	diagrams
	\[
	\diagram{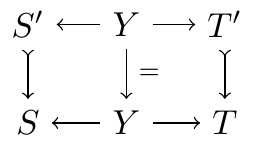}
	\quad
	\diagram{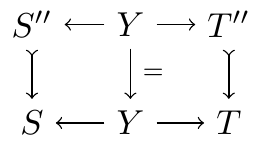}
	\qedhere
	\]
\end{proof}

The interchange law requires that, given composable spans of cospans
\begin{equation}
	\label{diag.2cells interchanged}
	\diagram{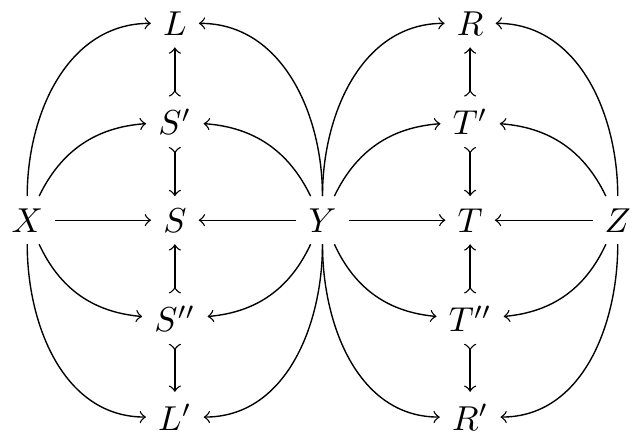}
\end{equation}
there is an isomorphism:
\begin{equation} 
\label{eq.interchange equation}
\left( S' \circ_\t{v} S'' \right) \circ_\t{h} 
\left( T' \circ_\t{v} T'' \right) \cong
\left( S' \circ_\t{h} T' \right) \circ_\t{v} 
\left( S'' \circ_\t{h} T'' \right).
\end{equation}
The left hand side corresponds to first applying vertical composition then horizontal composition.  The right hand side swaps the order of composition. This isomorphism will later strengthen to an equality when isomorphism classes of spans of cospans are the $2$-cells of a bicategory.

It is straightforward, using \eqref{eq.VertComp} and \eqref{eq.HorComp}, to see that \eqref{eq.interchange equation} reduces to finding an isomorphism
\begin{equation}
\label{eq.interchange simplified}
	(S' \times_S S'') +_Y (T' \times_T T'')
	\cong
	(S' +_Y T') \times_{S+_YT} (S'' +_Y T'')
\end{equation}
of spans of cospans.

%
%
%
%
%
%---------INTERCHANGE ISOMORPHISM-----------

To simplify our notation, write:
\begin{gather*}
A \coloneqq (S' \times_SS'') + (T' \times_TT''), \quad 
A_Y \coloneqq (S' \times_SS'') +_Y (T' \times_TT''),\\
B \coloneqq (S'+T') \times_{S+T} (S''+T''),  \quad  
B_Y \coloneqq (S'+_YT') \times_{S+_YT} (S''+_YT''). 
\end{gather*}
Now, apply Lemma \ref{lem.helpful little lemma} to the diagram
\[
	\label{NoRef_GettingAMaps}
	\diagram{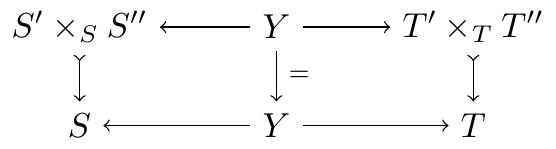}
\]
to get the pushout
\[
	\label{NoRef_AMaps}
	\diagram{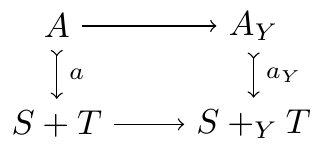}
\]
Similarly, we get pushouts
\[
	\diagram{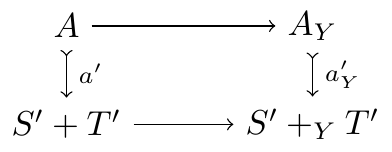}
	\quad
	\diagram{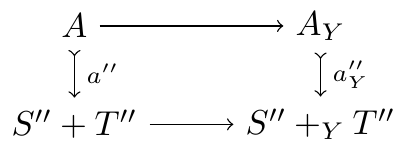}
\]
Now, $A$  forms a cone over the cospan $S+T \from S'+T' \cospan S'' + T''$ via the maps $a$, $a'$, and $a''$. And so, we get a canonical map $\theta \from A \to B$.  
%
%
%
%
%
%
%-----------LEMMA--------------

\begin{lem}
	\label{lem:pullback over subobject}
	Given cospans $Y$, $W \from X \cospan Z$ where the legs of $W$ factor through a monic $Y \rightarrowtail W$, then there is a unique isomorphism $X \times_Y Z \cong X \times_W Z$. 
\end{lem}
\begin{proof}
	Via the projection maps, $X \times_Y Z$ forms a cone over the cospan $W \from X \cospan Z$ and, also, $X \times_W Z$ forms a cone over the cospan $Y \from X \cospan Z$, though the latter requires the monic $Y \rightarrowtail W$ to do so. Universality implies that the induced maps are mutual inverses and they are the only such pair.  
\end{proof}
%
%
%
%
%
%
%--------------LEMMA-----------------

\begin{lem}
\label{lem.Theta Iso}
	The map $\theta \from A \to B$ is an isomorphism.
\end{lem}
\begin{proof}
	Because colimits are stable under pullback \cite[Thm.~4.7.2]{MacLaneMoerdijk_SheavesGeomLogic}, we get an isomorphism
	\[
	\gamma \from (S'\times_{S+T}S'') +(S'\times_{S+T}T'') +(T'\times_{S+T}S'') +(T'\times_{S+T}T'') \to B.
	\]
	But $S'\times_{S+T}T''$ and $S''\times_{S+T}T'$ are initial. To see this, recall that in a topos, all maps to the initial object are isomorphisms. Now, consider the diagram
	\[
		\label{NoRef_ShowingInitialObject}
		\diagram{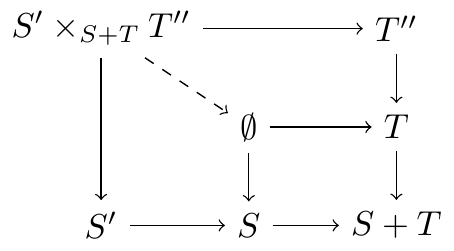}
	\]
	whose lower right square is a pullback because coproducts are disjoint in topoi.  Similarly, $T'\times_{S+T}S''$ is initial.  Hence we get a canonical isomorphism
	\begin{equation} \label{eq:B second iso}
	\gamma' \from (S'\times_{S+T}S'')+(T'\times_{S+T}T'') \to B
	\end{equation}
	that factors through $\gamma$. But Lemma \ref{lem:pullback over subobject} 
	gives unique isomorphisms $S' \times_{S} S'' \cong S' \times_{S+T} S''$ and $T'\times_{T} T'' \cong T' \times_{S+T} T''$. This produces a canonical isomorphism 
	\[
	\gamma'' \from A \to (S'\times_{S+T}S'')+(T'\times_{S+T}T'').
	\]
	One can show that $\theta = \gamma' \circ \gamma''$ using universal properties.  
\end{proof}
%
%
%
%
%
%
%---------GET THETA-Y ----------
Now, let us consider the following diagram:
\begin{equation}
	\label{diag.the big cube}
	\diagram{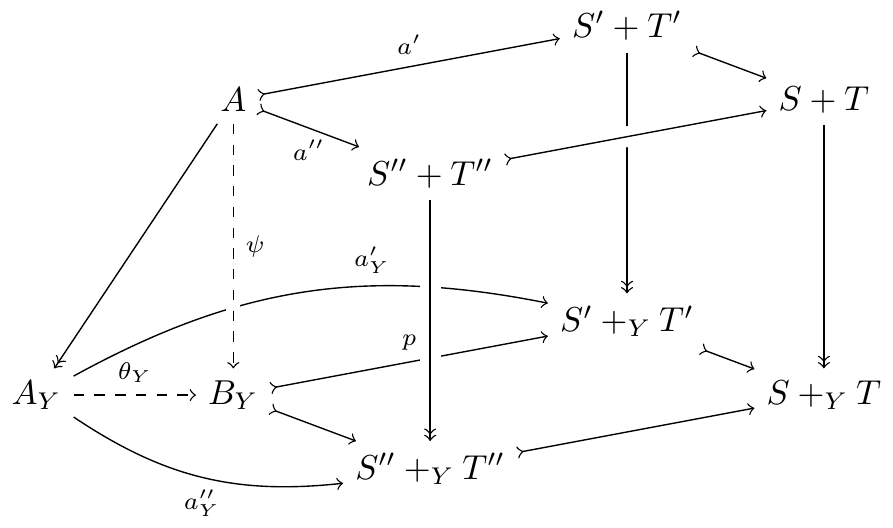}
\end{equation}
where $\theta_Y$ and $\psi$ are the canonical maps. Observe that $\psi$ factors through $\theta_Y$ in the above diagram.  This follows from the universal property of pullbacks. We also have that the top square is a pullback from the previous lemma.
%
%
%
%
%
%
%----------LEMMA---------------------

\begin{lem} 
	\label{lem.theta_Y iso}
	The map $\theta_Y \from A_Y \to B_Y$ is an isomorphism.
\end{lem}
\begin{proof}
	Because we are working in a topos, it suffices to show that $\theta_Y$ is both monic and epic. It is monic because $a'_Y$ is monic.
	
	To see that $\theta_Y$ is epic, it suffices to show that $\psi$ is epic. The front and rear right faces of \eqref{diag.the big cube} are pushouts by Lemma \ref{lem.helpful little lemma}.  Then because the top and bottom squares of \eqref{diag.the big cube} are pullbacks consisting of only monomorphisms, Lemma \ref{lem.vk dual} implies that the front and rear left faces are pushouts.  However, as pushouts over monos, Lemma \ref{lem.adhesive properties} tells us they are pullbacks.  But in a topos, regular epis are stable under pullback, and so $\psi$ is epic.  	
\end{proof}
%
%
%
%
%
%
%---------ISO OF 2CELLS ----------------

It remains to show that $\theta_Y$ serves as 
an isomorphism between spans of cospans. This 
amounts to showing that
\begin{equation} 
	\label{diag.theta 2-cell iso}
	\diagram{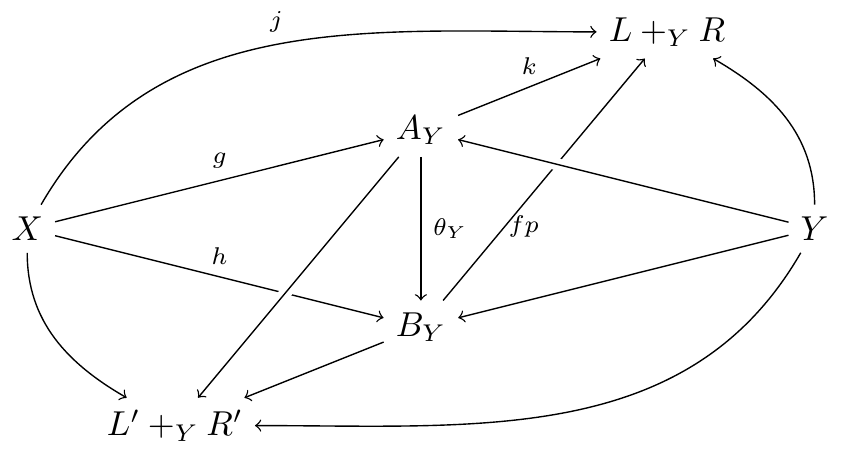}
\end{equation}
commutes. Here $g$ and $k$ are induced from applying vertical composition before horizontal, $h$ from applying horizontal composition before vertical, $j$ is from composing in either order, $f$ is from \eqref{eq.HorComp}, and $p$ is from \eqref{diag.the big cube}.  The top and bottom face commute by construction.
%
%
%
%
%
%
%-----------LEMMA----------------------

\begin{lem}
	The inner triangles of diagram \eqref{diag.theta 2-cell iso} commute. That is, we have $k=fp \theta_Y$ and $h=\theta_Yg$.
\end{lem}
\begin{proof}
	To see that $k=fp\theta_Y$, consider the diagram
	\[
		\label{NoRef_2CellIso}
		\diagram{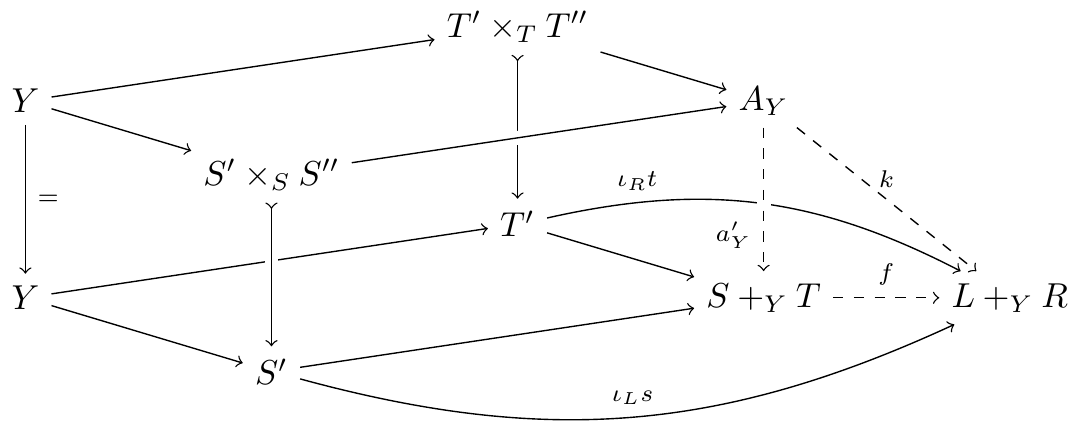}
	\]
	The bottom face is exactly the pushout diagram from which $f$ was obtained.  Universality implies that $k = f a'_Y$ and, as seen in \eqref{diag.the big cube}, $a'_Y = p \theta_Y$. 
	
	That $h=\theta_Yg$ follows from 
	\[
		fph=j=kg=fp\theta_Yg
	\] 
	and the fact that $fp$ is monic.
\end{proof}

%
%
%
%
%
%------------------------INTERCHANGE LAW -----------------------------------

Of course, we have only shown that two of the four inner triangles commute, but we can replicate our arguments to show the remaining two commute as well.  This lemma was the last step in proving the following interchange law.

\begin{thm}
	\label{thm.interchange law}
	Given diagram \eqref{diag.2cells interchanged} in a topos, there is a canonical isomorphism $(S' \times_S S'') +_Y (T' \times_T T'') \cong (S' +_Y T') \times_{S +_Y T} ( S'' +_Y T'')$.
\end{thm}

%
%
%
%
%
%%%%%%%%%%%%%%%%%%%%%%%%%%%%%%%%%%%%%%%%%%%%%%%%%%%%%%%%%%%%%%%%%%
%
\subsection{Constructing the bicategory}  %            THE BICATEGORY
%
%%%%%%%%%%%%%%%%%%%%%%%%%%%%%%%%%%%%%%%%%%%%%%%%%%%%%%%%%%%%%%%%%%
%
%
%
%
%
%
%--------------COSPAN 2* C----------------

Let $\cat{C}$ be any topos. We will commence construction of a bicategory named $\cat{MonSp(Csp(C))}$, or $\csC$ for short. The $0$-cells of $\csC$ are just the $\cat{C}$-objects. For $0$-cells $X$ and $Y$, build a category $\csC(X,Y)$ whose objects are $\cat{C}$-cospans and morphisms are isomorphism classes of $\cat{C}$-spans of cospans whose legs are both monic. Composition in $\csC (X,Y)$ is the vertical composition $\circ_v$ introduced in \eqref{eq.VertComp}. It is straightforward to check that associativity holds and that spans of cospans whose legs are identity serve as identities.
%
%
%
%
%
%
%------------COMPOSITION FUNCTOR---------

The composition functor is given by an assignment
\[
\otimes \from \csC(Y,Z) \times \csC(X,Y) \to \csC(X,Z)
\]
that acts on $1$-cells by $(T,S) \mapsto S \otimes_Y T$ and on $2$-cells by horizontal composition $\circ_h$ from \eqref{eq.HorComp}. It is straightforward to check that $\otimes$ preserves identities. Theorem \ref{thm.interchange law} ensures that $\otimes$ preserves composition.
%
%
%
%
%
%------------IDENTITY & ASSOCIATOR % UNITOR FUNCTOR-----------

For every $0$-cell $X$, the identity functor $\cat{1} \to \csC (X,X)$ picks out the $2$-cell with all identity maps on $X$. The associator is made of $2$-cells 
\[
R+_XS+_YT \from (R+_XS)+_YT \span R+_X(S+_YT).
\] 
The right unitor is made of $2$-cells $S \from S+_YY \span S$. Likewise, the left unitor has $2$-cells $T \from T \span Y+_YT$. The legs for each of the above are the obvious choices. The pentagon and triangle identities follow from the associativity, up to isomorphism, of pushouts. 
%
%
%
%
%
%-----------MAIN THEOREM--------------

Given all of the data just laid out, we have the main theorem of the paper.

\begin{thm}
	If $\cat{C}$ is a topos, then $\cat{MonSp(Csp(C))}$ is a bicategory.
\end{thm}
%
%
%
%
%
%
%
%%%%%%%%%%%%%%%%%%%%%%%%%%%%%%%%%%%%%%%%%%%%%%%%%%%%%%%%%%%%%%%%%%%%%%%%
%
\section{A discussion on the assumptions}  %           THE ASSUMPTIONS
\label{sec.Disc on Assump}
%
%%%%%%%%%%%%%%%%%%%%%%%%%%%%%%%%%%%%%%%%%%%%%%%%%%%%%%%%%%%%%%%%%%%%%

Can we expand the domain on which this construction works? Apart from ensuring sufficiently many limits and colimits, the primary source of roadblocks is the interchange law. As we discuss below, it is not absolutely necessary to work strictly within a topos, but we do so in order to be expeditious.  To lay out, one by one, the requirements for our interchange law to hold would be exhausting and leave us little energy to work through its proof. To do this is even less reasonable given that we can just shout ``topos'' and move on.   However, listing these requirements is interesting enough to take a look at here.  
%
%
%
%
%
%------------------NON-EXAMPLE  FOR MONIC LEGS-------------------------

Before digging deeper into the properties used, let's convince ourselves of the necessity for monic legs within our span of cospans.

\begin{ex}
	Consider the category $\cat{Set}$ of sets and functions. We will relax the assumption that the legs of the spans of cospans are both monic.  Indeed, suppose that $S'$, $S''$, and $T'$ are two element sets and $S$, $Y$, $T$, and $T''$ are singletons.  The functions can be any of the limited choices we have.  After several routine calculations, we determine that $(S' \times_S S'') +_Y (T' \times_T T'')$ has cardinality $5$ and $(S' +_Y T') \times_{S+_YT} (S''+_YT'')$ has cardinality $6$. 
\end{ex}

So we see that even in the archetypal topos, the legs of the spans must be monic for the interchange law to hold. But even if they are, this law may fail if our category $C$ is not a topos. The next example illustrates this.

\begin{ex}
	Consider the Boolean algebra on a two element set.  This is the category $0 \to 1$ with products given by meet and coproducts given by join. Note that this is not a topos. Indeed, the only non-identity morphism is both monic and epic but, as no inverse exists, it is not is isomorphism. Recalling the interchange equation \eqref{eq.interchange equation}, suppose that we have $Y=S''=T'=0$ and $S=S'=T=T''=1$.  It is straightforward to check that $(S' \times_S S'') +_Y (T' \times_T T'') = 0$ and $(S' +_Y T') \times_{S+_YT} (S''+_YT'')=1$. That is, the interchange equation does not hold.  Because this Boolean algebra can be embedded into any other, it follows that interchange does not hold for any Boolean algebra.
\end{ex}

Now that we are convinced that we actually do need to assume something for interchange to work, we can list our requirements.  The obvious place to start is by asking for our category to have enough limits and colimits.  

In Lemma \ref{lem.helpful little lemma}, we use that pushouts respect monics.  This occurs in topoi. Indeed, this occurs in adhesive categories which are discussed in Section \ref{sec.Rewriting}.  

In Lemma \ref{lem.Theta Iso}, we use a couple of facts.  First, we use that coproducts are disjoint, which means that a pullback over coproduct inclusions is initial. We also use that colimits -- particularly coproducts -- are stable under pullback.  The full force of a topos is not needed to satisfy this requirement.  Indeed, looking at \cite[Thm.~1.4.9]{MacLaneMoerdijk_SheavesGeomLogic}, we simply need our category to be locally cartesian closed.  

In Lemma \ref{lem.theta_Y iso}, we use that a monic epimorphism is an isomorphism in topoi. This is surely not true in general, but is is true that a monic regular epimorphism is always an isomorphism.  Notice that the vertically aligned epimorphisms in diagram \eqref{diag.the big cube} are all regular since they are all coequalizers. For instance $S+T \twoheadrightarrow S+_YT$ is a coequalizer over $Y \rightrightarrows S+T$.  So we can merely ask for pullbacks to preserve regular epimorphisms. Note that this does happen in a regular category, which is a larger class than topoi.   

Also in Lemma \ref{lem.theta_Y iso}, we make use of Lemmas \ref{lem.adhesive properties} and \ref{lem.vk dual}.  These holds in any adhesive category. 

It is clear that the properties of adhesive categories play an important role in ensuring that $\csC$ is a bicategory, as does being locally cartesian closed and regular.  Certainly, topoi are a well-known, large family of categories that are contained in the intersection of all of these classes. 
%
%
%
%
%
%
%
%
%
%
%
%
%
%
%
%
%
%
%%%%%%%%%%%%%%%%%%%%%%%%%%%%%%%%%%%%%%%%%%%%%%%%%%%%%%%%%%%%%%%%%%%%%%
%
\section{An Informal Introduction to Graph Rewriting}  %          GRAPH REWRITING
\label{sec.Rewriting}
%
%%%%%%%%%%%%%%%%%%%%%%%%%%%%%%%%%%%%%%%%%%%%%%%%%%%%%%%%%%%%%%%%%%%%%%%%%

For those who are not familiar with rewriting systems, and graph rewriting in particular, we provide a brief introduction. For a more in-depth and rigorous viewpoint, see \cite{Baader_TermRewritingAllThat}, \cite{Ehrig_GraphGramAlgAp}, or \cite{LackSoboc_AdhesiveCategories}.

There are many methods of rewriting found throughout mathematics, computer science, and linguistics. The general idea is that we begin with a collection of rules, a collection of terms and a way to apply rules to certain, compatible terms.  When applied to a term, a rule replaces a sub-term with a new sub-term. A simple, very informal example is found within the English language.  Consider a set of rules 
\[
\{ \t{(noun)} \mapsto x : x \t{ is an 
	English noun} \}.
\]
We can apply any one of these rules to
\[
\t{`The (noun) is behind you.'}
\] 
to obtain heaps of grammatically correct, if potentially spooky, sentences.  

The first rewriting methods were uni-dimensional, in that they are concerned with replacing a string of characters or letters. Many attempts to define multi-dimensional rewriting systems came up short in application and execution, but Ehrig, Pfender, and Scheider developed a categorical approach using graph morphisms and pushouts  \cite{Ehrig_GraphGramAlgAp} that has since been studied extensively. This approach came to be known as \textit{double pushout graph rewriting}. This is what we are interested in and will consider in our bicategory $\cat{Rewrite}$ introduced below.

Here is how double pushout graph rewriting works. A \emph{production} is a span 
\[
p: L \leftarrowtail K \rightarrowtail R
\] 
of graphs with monic legs. Some authors call this a `linear production' to distinguish from spans with potentially non-monic legs, but we will not adopt this convention here. Given a production $p$, and a graph morphism $L \to C$, called a \emph{matching map}, such that there exists a diagram
\begin{equation}
	\label{diag:derivation}
	\diagram{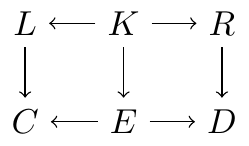}
\end{equation}
consisting of two pushout squares, we say that that $D$ is a \emph{direct derivation} of $C$ and write $C \rightsquigarrow_E D$ or just $C \rightsquigarrow D$. It is common enough to decorate the arrow `$\rightsquigarrow$' with more information: for instance, the name of the production or the matching map.  But that will not be necessary here. Observe that the objects $E$ and $D$ need not exist, but when they do, they are unique up to isomorphism \cite[Lemma 4.5]{LackSoboc_AdhesiveCategories}. 

A \emph{grammar} $(\mathcal{G},\mathcal{P})$ is a set of graphs $\mathcal{G}$ paired with a 
set of productions $\mathcal{P}$. A \textit{derivation of the grammar} is a string of direct derivations 
\[
	G_0 \rightsquigarrow G_1 \rightsquigarrow \dotsm \rightsquigarrow G_n
\] 
from productions in $\mathcal{P}$ and with $G_0 \in \mathcal{G}$. We say that $G_n$ is a \emph{rewrite of $G_0$}. The \textit{language} $\mathcal{L}(\mathcal{G},\mathcal{P})$ generated by the grammar is the collection of all graphs $G$ such that there is a derivation $G_0 \rightsquigarrow^\ast G$ of the grammar. The idea is that one will study a language. Exactly what properties are interesting is beyond the scope of this discussion.  Interested readers should consult the references mentioned at the beginning of this section.  Instead of going deeper into the subject, we will briefly zoom out.
%
%
%
%
%
%
%-----------ADHESIVE BLURB--------------

Searching for a general framework for term graph rewriting, Lack and Sobocinski \cite{LackSoboc_AdhesiveCategories} introduced a class of categories they call \emph{adhesive}. Roughly, a category is adhesive if it has pullbacks, pushouts along monomorphisms, and certain exactness conditions between pullbacks and pushouts hold. This is not a trivial class of categories given that topoi are adhesive \cite{LackSoboc_ToposesAdhesive}. Because of this, we were able to use the following lemmas, which were proven for adhesive categories, in our construction. However, we will just present them for topoi.                                 
%
%
%
%
%
%
%----------LEMMA---------------

\begin{lem}[{\cite[Lemmas 
		4.2-3]{LackSoboc_AdhesiveCategories}}]
	\label{lem.adhesive properties}
	In a topos, monomorphisms are stable under 	pushout. Also, pushouts along monomorphisms 
	are pullbacks.
\end{lem}
%
%
%
%
%
%
%----------LEMMA---------------

\begin{lem}[{\cite[Lemma 
		6.3]{LackSoboc_AdhesiveCategories}}] 
	\label{lem.vk dual}
	In a topos, consider a cube
	\[
		\label{NoRef_AdhsiveDualCube}
		\diagram{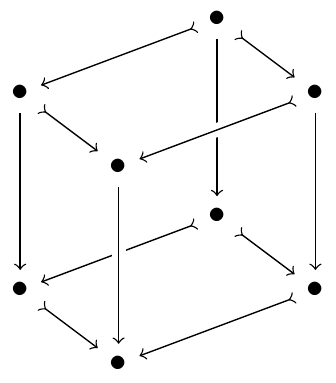}
	\]
	whose top and bottom faces consist of only monomorphisms. If the top face is a pullback 
	and the front faces are pushouts, then the bottom face is a pullback if and only if the 
	back faces are pushouts.
\end{lem}

A good portion of the theory for double pushout graph rewriting has been extended to adhesive categories. So, while our focus is and will be on double pushout graph rewriting, there may be variations of the bicatgory $\cat{Rewrite}$ (introduced below) that are of interest to computer scientists.  For instance, the \emph{Schanuel topos} was used to model the $\pi$-calculus \cite{Fiore_OpModelsProcessCalulii} and adhesive categories allow  us to extend rewriting to such settings.  Our contribution is a bicategorical framework to house the rewriting as $2$-cells, though only for double pushout graph rewriting.

%%%%%%%%%%%%%%%%%%%%%%%%%%%%%%%%%%%%%%%%%%%%%%%%%%%%%%%%%%%%%%%%%%%%%%%%%
\subsection{$\cat{Rewrite}$}  %                REWRITE
\label{sec.Rewrite}
%%%%%%%%%%%%%%%%%%%%%%%%%%%%%%%%%%%%%%%%%%%%%%%%%%%%%%%%%%%%%%%%%%%%%%%%%

Here, we introduce the bicategory $\cat{Rewrite}$ as promised.  To prepare, we begin by introducing a slight generalization of the double pushout graph rewriting concepts discussed above.  
An \emph{interface} $(I,O)$ is a pair of discrete graphs and a \emph{production with interface} $(I,O)$, or simply \emph{$(I,O)$-production}, is a cospan of spans
\[
	\label{NoRef_IOProduction}
	\diagram{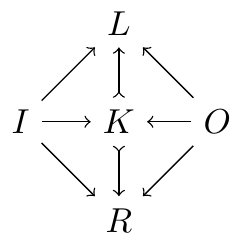}
\]
Think of $I$ and $O$ as choosing inputs and outputs. Given a production with interface $(I,O)$, we say that a graph $G'$ is a \emph{direct $(I,O)$-derivation} of $G$ if there is a diagram
\[
	\label{NoRef_IODerivation}
	\diagram{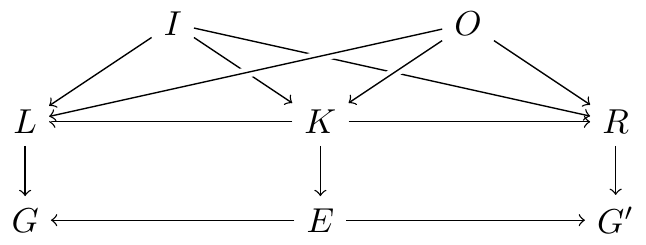}
\]
where the bottom squares are pushouts. We denote this as $G \rightsquigarrow G'$. An \emph{$(I,O)$-grammar} $(\mathcal{G},\mathcal{P})$ consists of a collection of graphs $\mathcal{G}$ and another of $(I,O)$-productions $\mathcal{P}$. Again, an $(I,O)$-grammar generates a language consisting of all graphs $G$ such that there is a chain of direct $(I,O)$-derivations $G_0 \rightsquigarrow G_1 \rightsquigarrow \dotsm \rightsquigarrow G_n=G$ from $\mathcal{P}$ such that $G_0 \in \mathcal{G}$. However, this time, we require such a chain to respect the inputs and outputs in the sense that 
\begin{equation}
	\label{diag.DerivationChain}
	\diagram{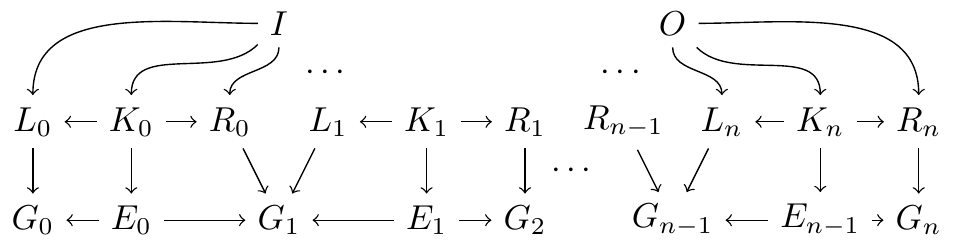}
\end{equation}
Now we can put these productions with interfaces into our bicategorical framework.

Consider the full sub-bicategory of $\csC$ for $\cat{C} := \cat{Graph}$ with objects the finite, discrete graphs. Here, a $1$-cell is a cospan $G \from I \cospan O$ where $G$ is a graph and $I$, $O$ are discrete graphs.  A $2$-cell between $1$-cells $G'$ and $G''$ is a span of cospans $G \from G' \span G''$ which we think of as an $(I,O)$-derivation $G' \rightsquigarrow G''$.  We will call this sub-bicategory $\cat{Rewrite}$ because the $2$-cells are exactly all the possible ways to rewrite one graph into another so that inputs and outputs are preserved. Given any $2$-cell $G \from G' \span G''$ in $\cat{Rewrite}$, then the diagram 
\[
	\label{NoRef_2CellToRewrite}
	\diagram{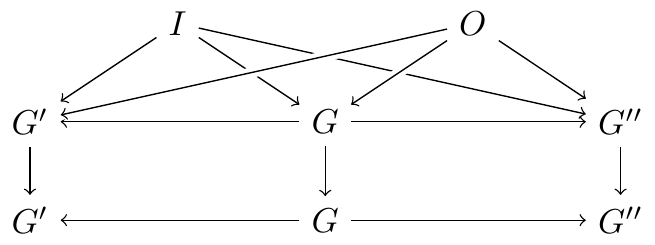}
\]
gives an $(I,O)$-derivation $G' \rightsquigarrow_G G''$. Conversely, any $(I,O)$-derivation \eqref{diag.DerivationChain} can be made into composable $2$-cells $E_i \from G_i \span G_{i+1}$ where the maps from $I$ and $O$ are the evident composites. Then, the vertical composition of the resulting $2$-cells gives us the desired span of cospans.

To better illustrate this, we will provide a concrete example of the dictionary between $\cat{Rewrite}$ and double pushout graph rewriting.  Suppose we were given an $(I,O)$-derivation, with $I=\{\ast\}=O$, induced from the following double pushout graph rewriting diagram
\[
	\label{NoRef_DPOGraphRewritingExample}
	\diagram{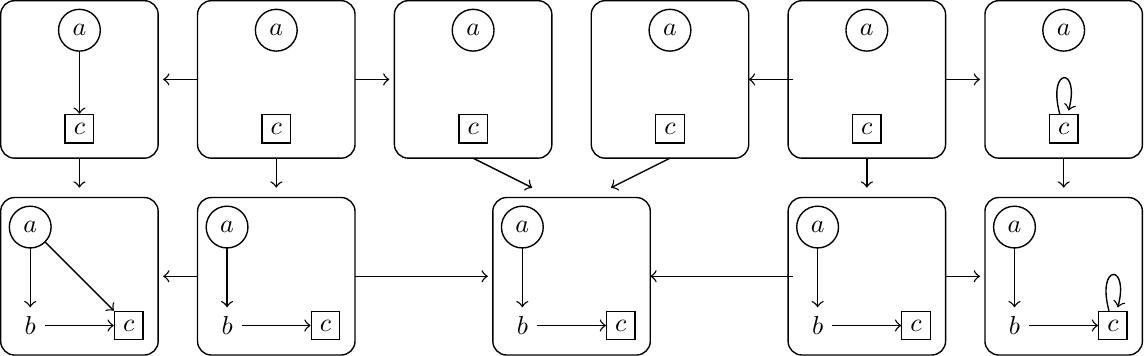}
\]
where the functions are described by the labeling, the inputs are circled, and the outputs are squared. In words, we have rewritten the graph on the lower left by remove an edge $a \to c$ and adding a loop $c \to c$. The pull back of the span
\[
	\label{NoRef_DPOGraphRewritingExample2}
	\diagram{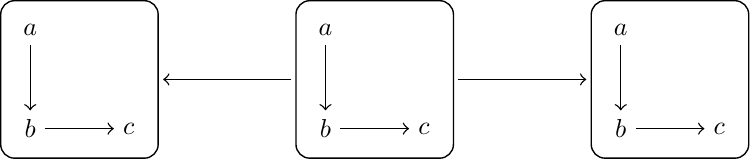}
\]
is the graph
\[
	\label{NoRef_DPOGraphRewritingExample3}
	\diagram{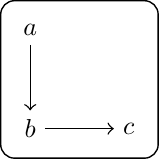}
\]
The corresponding $2$-cell is the diagram
\[
	\label{NoRef_DPOGraphRewritingExample4}
	\diagram{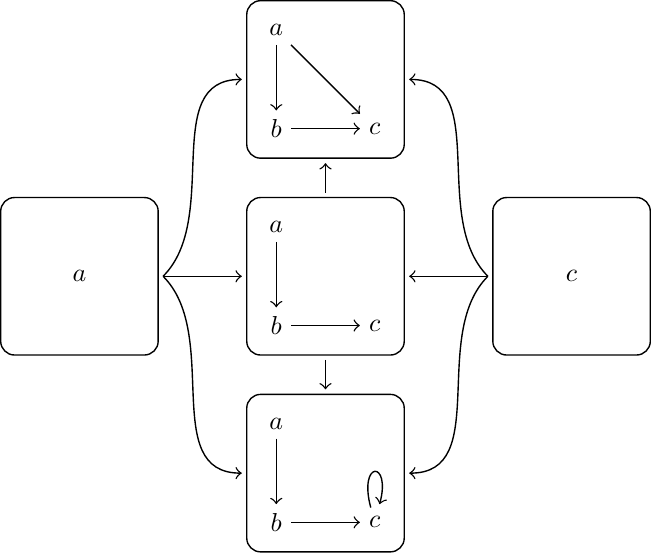}
\]
Here we witness the advantage that graph rewriting has over graph morphisms in the realm of expresivity. There is no way to replace the $2$-cell above with a map of graph cospans. 

There is nothing inherently special, from a mathematical point of view, about working with graphs and their morphisms.  It is in applications where graphs gain importance.  We can actually create categories analogous to $\cat{Rewrite}$ with any topos.  
%
%
%
%
%
%
%%%%%%%%%%%%%%%%%%%%%%%%%%%%%%%%%%%%%%%%%%%%%%%%%%%%%%%%%%%%%%%%%%%%%%%%%%%
\section{Conclusion and further work}  %                      CONCLUSION
%%%%%%%%%%%%%%%%%%%%%%%%%%%%%%%%%%%%%%%%%%%%%%%%%%%%%%%%%%%%%%%%%%%%%%%%%

Our primary motivation for constructing this bicategory is as a way to study the gluing of graphs together in a way compatible with chosen input and output nodes. To this end, we defined a bicategory $\cat{Rewrite}$. This is a very large bicategory and in practice, one begins with a grammar and studies the resulting language. So, in an upcoming collaboration with Kenny Courser, we will look at relating languages to sub-bicategories of $\cat{Rewrite}$ generated by a grammar. In this same paper, we will study the structure of $\cat{MonSp(Csp(C))}$ alongside similar bicategories.

\end{document}